\documentclass[11pt]{article}

\usepackage[charter]{mathdesign}  \usepackage{fontspec}
  \usepackage[margin=1in]{geometry}
\usepackage[square,comma,sort&compress]{natbib}
\usepackage{graphicx}
\usepackage{booktabs}
\usepackage{tabularx}
\usepackage{array}
\usepackage{amsmath}
\usepackage{amsthm}
\usepackage{mathtools}
\usepackage{enumitem}
\usepackage{xcolor}
\usepackage{xurl}
\usepackage{aliascnt}
\usepackage{hyperref}
\usepackage[capitalize,nameinlink,noabbrev]{cleveref}
\usepackage[disable]{todonotes}

\definecolor{darkblue}{rgb}{0.0,0.0,0.45}
\hypersetup{
  colorlinks=true,
  linkcolor=darkblue,
  citecolor=darkblue,
  urlcolor=darkblue
}

\theoremstyle{plain}
\newtheorem{theorem}{Theorem}[section]
\newaliascnt{lemma}{theorem}
\newtheorem{lemma}[lemma]{Lemma}
\aliascntresetthe{lemma}
\newaliascnt{proposition}{theorem}
\newtheorem{proposition}[proposition]{Proposition}
\aliascntresetthe{proposition}
\newaliascnt{corollary}{theorem}

\aliascntresetthe{corollary}
\theoremstyle{definition}
\newaliascnt{definition}{theorem}
\newtheorem{definition}[definition]{Definition}
\aliascntresetthe{definition}
\newaliascnt{question}{theorem}
\newtheorem{question}[question]{Question}
\aliascntresetthe{question}
\newaliascnt{remark}{theorem}
\newtheorem{remark}[remark]{Remark}
\aliascntresetthe{remark}

\crefname{theorem}{Theorem}{Theorems}
\Crefname{theorem}{Theorem}{Theorems}
\crefname{lemma}{Lemma}{Lemmas}
\Crefname{lemma}{Lemma}{Lemmas}
\crefname{proposition}{Proposition}{Propositions}
\Crefname{proposition}{Proposition}{Propositions}
\crefname{corollary}{Corollary}{Corollaries}
\Crefname{corollary}{Corollary}{Corollaries}
\crefname{definition}{Definition}{Definitions}
\Crefname{definition}{Definition}{Definitions}
\crefname{question}{Question}{Questions}
\Crefname{question}{Question}{Questions}
\crefname{remark}{Remark}{Remarks}
\Crefname{remark}{Remark}{Remarks}

\DeclareMathOperator{\conv}{conv}
\DeclareMathOperator{\rank}{rank}
\newcommand{\R}{\mathbb{R}}
\newcommand{\one}{\mathbf{1}}
\newcommand{\bits}{\{0,1\}}

\newcommand{\unitvec}{\mathbf{e}}

\title{A Counterexample to Ziegler's Cross-Polytope Conjecture \\ for Simplicial \texorpdfstring{$0/1$}{0/1}-Polytopes}

\author{Volker Kaibel\\
Faculty of Mathematics, Otto-von-Guericke-Universit\"at Magdeburg, Germany\\
\texttt{kaibel@ovgu.de}
\and
Sebastian Pokutta\\
Institute of Mathematics, Technische Universit\"at Berlin and\\
Zuse Institute Berlin, Germany\\
\texttt{pokutta@zib.de}}

\date{June 30, 2026}

\begin{document}
\maketitle

\begin{abstract}
Ziegler proved that every simplicial $d$-dimensional $0/1$-polytope has at most
$2d$ vertices, and asked whether equality forces the polytope to be centrally
symmetric and hence, equivalently, a $0/1$-realization of the $d$-dimensional cross
polytope. In this note, we give a negative answer, exhibiting an explicit set of $14$
vertices in $\bits^7$ whose convex hull is a simplicial $7$-polytope and is not centrally symmetric.  
Moreover, via exhaustive enumeration we show that
up to the
symmetries of the cube, there are precisely five such polytopes in
dimension~$7$ (of two combinatorial types) 
that are not centrally symmetric.
\end{abstract}

\section{Introduction}
\label{sec:introduction}

A $0/1$-polytope is the convex hull of a subset of the vertices of the unit
cube.  These polytopes form a central class in polyhedral combinatorics and they
include many polytopes arising from discrete optimization. Several central questions revolve around the class of $0/1$-polytopes. In this work, we are interested in simplicial $d$-dimensional $0/1$-polytopes, i.e., those ones  where every at most $(d-1)$-dimensional face is a simplex. 
Such polytopes can only have a linear number 
\[
  f_0(P)\le 2d
\]
 of vertices (i.e.,  $0$-dimensional faces), see Proposition~\ref{prop:ziegler-bound}. 
\citet{Ziegler1999} asked whether equality fully determines the structure of the polytope:

\begin{question}[\citet{Ziegler1999}]
\label{ques:ziegler}
Is every simplicial $d$-dimensional $0/1$-polytope with $2d$ vertices affinely isomorphic to the cross
polytope $\conv\{\pm\unitvec_1,\dots,\pm\unitvec_d\}$ (with $\unitvec_i=(0,\dots,0,1,0,\dots,0)$?  Equivalently (see Lemma~\ref{lem:central-cross}), is every such equality case centrally symmetric?
\end{question}

Note that central symmetry of a 0/1-polytope means that its vertex set consists of $d$ opposite pairs, i.e., pairs of the form $(v, \one - v)$ with $v$ a $0/1$-point. In his lecture notes, \citet{Ziegler1999} writes that Aichholzer had enumerated all $6$-dimensional
$0/1$-polytopes with at most $12$ vertices and \Cref{ques:ziegler} is true for $d\le6$; see also the systematic study of extremal $0/1$-polytopes in
\citet{KortenkampRichterGebertSarangarajanZiegler1997,Aichholzer2000}. As a byproduct of our classification later in \Cref{sec:enumeration}, we will also verify the affirmative answer to the question for $d \leq 6$. 

\paragraph{Contribution.}
We provide an explicit example establishing  the negative answer to \Cref{ques:ziegler}, given as the convex hull of the $14$ binary vectors listed in \Cref{thm:main}.  The
polytope is $7$-dimensional and every facet is
a $6$-simplex, yet four vertices $v$ do not have their cube antipode $\one - v$ present. Thus central symmetry fails even though the hypotheses of \Cref{ques:ziegler} are satisfied. Given that the conjecture is true for $d \leq 6$, this is the first possible dimension for such a counterexample. Additionally, we performed an exhaustive search reported in \Cref{sec:enumeration} providing a full characterization of such polytopes in dimension $7$: up to the symmetry group of the cube there are exactly five examples, in two combinatorial types, that are not centrally symmetric.

The constructions have been verified by means of two
independent implementations. The first enumerates (candidate) facet defining hyperplanes and checks all linear-algebra assertions over the rational numbers.  The second uses
\texttt{polymake}~\citep{polymake} and its \texttt{SIMPLICIAL} predicate, again over rational arithmetic, to compute
the dimension, vertices, facets, incidences, $f$-vector, $h$-vector, and graph.
The two computations agree.

\paragraph{Related work.}
The systematic investigation of  $0/1$-polytopes was initiated by
\citet{KortenkampRichterGebertSarangarajanZiegler1997} and further developed in
Ziegler's survey~\citep{Ziegler1999}.  Subsequently, the knowledge about that class of polytopes has been increased considerably. For instance, simple (i.e., non-degenerate) 0/1-polytopes have been characterized by \citet{KaibelWolff2000}, bounds for the numbers of facets have been derived by 
\citet{FleinerKaibelRote2000, BaranyPor2001}, and structural insights into random 0/1-polytopes, in particular their graphs, have been obtained, e.g. by \citet{Ferber_etal2026}, \citet{Babecki_etal2025}, and \citet{Kaibel2004}.

\section{The conjecture and the equality case}
\label{sec:preliminaries}

Let $V\subseteq\bits^d$ and let $P = \conv(V)\subseteq\R^d$.  We denote the all-one vector by $\one=(1,\ldots,1)$ and call $v\mapsto\one-v$ the \emph{cube antipodal map}, which assigns to a $0/1$-vertex $v$ its antipode $\one - v$.  The set
$V$ is \emph{centrally symmetric with respect to the cube center} if
for every $v\in V$ we have $\one-v\in V$. A $d$-dimensional polytope $P$ is \emph{simplicial} if every facet of $P$ is a simplex of dimension $d-1$. The following elementary vertex bound motivates \Cref{ques:ziegler}. 

\begin{proposition}[Vertex bound; remark after Proposition 17 in \citet{Ziegler1999}]
\label{prop:ziegler-bound}
If $P\subseteq[0,1]^d$ is a simplicial $d$-polytope, then $f_0(P)\le 2d$.
Moreover, if equality holds, then for every coordinate $i$ both coordinate faces
$P\cap\{x_i=0\}$ and $P\cap\{x_i=1\}$ contain exactly $d$ vertices and are
$(d-1)$-simplices.
\end{proposition}

\begin{proof}
For a fixed coordinate $i$, the hyperplanes $x_i=0$ and $x_i=1$ support the cube
and hence cut out faces of $P$.  Since $P$ is simplicial, both faces are
simplices of dimension at most $d-1$, so each has at most $d$ vertices.  Every
vertex of $P$ has $i$th coordinate either $0$ or $1$, so the two faces partition
$V$ and $f_0(P)\le d+d=2d$.  If equality holds, both inequalities must be
tight for every coordinate.
\end{proof}

A useful consequence of \Cref{prop:ziegler-bound} that we exploited in both the search for the counterexample and the full characterization in \Cref{sec:enumeration} is that the vertex set is \emph{balanced} in every coordinate, i.e., in each coordinate bnoth values $0$ and $1$ occur exactly $d$ times, so for the barycenter  
$\frac1{2d}\sum_{v\in V}v = \frac12\one$ holds. As the barycenter
of the vertex set of a full-dimensional polytope is a proper convex
combination of all of its vertices, it follows that $\frac12\one\in\operatorname{int}P$. In particular every cube-antipodal
pair $\{v,\one-v\}$ has the interior point $\frac12\one$ as its midpoint, so that it cannot be an edge of the polytope. While this observation is not needed to state the
counterexample, it explains the numerical invariants below.

\section{The counterexample}
\label{sec:counterexample}

We will now provide our counterexample.

\begin{theorem}[Main theorem]
\label{thm:main}
Let $P=\conv(V)\subseteq[0,1]^7$, where
\[
\begin{aligned}
V=\{\,&0010110,\ 1011101,\ 1000100,\ 1001010,\ 0111000,\ 1100001,\ 0010001,\\
     &0001100,\ 0100010,\ 1001111,\ 1101110,\ 0110101,\ 1110011,\ 0111011\,\}.
\end{aligned}
\]
Then $P$ is a simplicial $7$-dimensional $0/1$-polytope with $14=2\cdot7$
vertices.  It is not centrally symmetric.  Consequently, not every simplicial
$d$-dimensional $0/1$-polytope with $2d$ vertices is a cross polytope.
\end{theorem}

\begin{proof}
Let $M_V$ be the matrix with rows consisting of $V$. All entries of $M_V$ are $0$ or $1$, and the $14$ rows are distinct. Computing the rank
\[
  \rank\begin{pmatrix} 1 & v_1\\ \vdots & \vdots\\ 1 & v_{14}\end{pmatrix}=8
\]
shows that $P$ is $7$-dimensional.  Exact facet enumeration gives $136$
supporting facets.  Each facet contains exactly $7$ vertices; since $\dim P=7$,
each facet is therefore a $6$-simplex.  Therefore $P$ is simplicial.  Every point of a
$0/1$-set is  a vertex of its convex hull,
so the $14$ distinct rows of $M_V$ are exactly the
vertices of $P$ and $f_0(P)=14$ follows.

It remains to observe that $V$ is not closed under the cube antipodal map. To this end observe that the four vertices
\[
  v_1=0010110,\qquad v_5=0111000,\qquad v_6=1100001,\qquad v_{10}=1001111
\]
have complements that are not in $V$.  Hence $P$ is not centrally symmetric.
\end{proof}

The facet enumeration and rank computations were carried out exactly over
\(\mathbb{Q}\); an independent \texttt{polymake} computation returns the same
dimension, number of vertices, facets, and simpliciality certificate.

\begin{figure}
  \centering
  \includegraphics[width=0.7\linewidth]{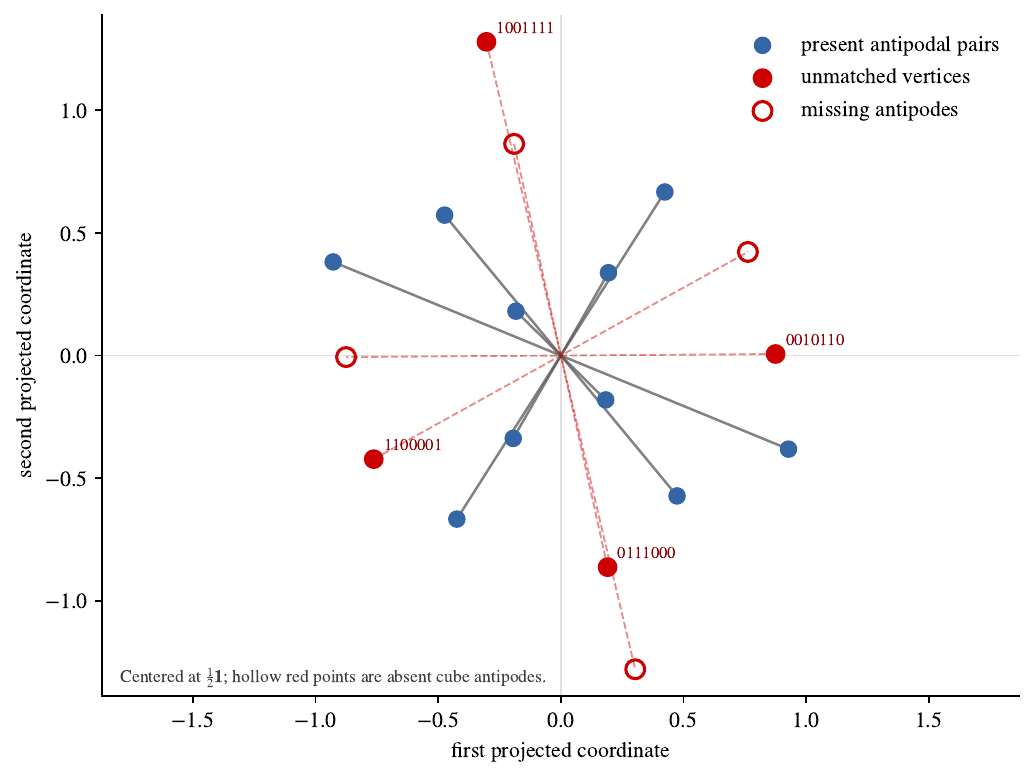}
  \caption{A two-dimensional projection, chosen to separate all points, of the centered vertex set $V-\frac12\one$.  Blue points form the five present cube-antipodal pairs.
  Red points are the four unmatched vertices; hollow red markers indicate where
  their missing cube antipodes would project. A single coordinate swap matches the remaining red vertices to obtain a cross polytope.}
  \label{fig:separating-projection}
\end{figure}

\begin{remark}[A single coordinate swap restores central symmetry]
\label{rem:single-swap}
Observe that a single coordinate swap turns our counterexample into a cross polytope.  Exchanging the fourth coordinates of $v_5=0111000$ and $v_6=1100001$ produces $v_5'=0110000$ and $v_6'=1101001$.  Now $v_5'=\one-v_{10}$ and $v_6'=\one-v_1$, so the four
previously unmatched vertices $v_1,v_5,v_6,v_{10}$ pair up and the resulting $14$
points are centrally symmetric about $\frac12\one$.  Their seven antipodal
directions $2v-\one$ are linearly independent, so the convex hull of the modified set is affinely isomorphic to 
the $7$-dimensional cross polytope (now with $2^7=128$ facets); see also \Cref{fig:separating-projection} for an illustration.
\end{remark}

Finally, we provide some combinatorial data for our counterexample.

\begin{remark}[Combinatorial data]
\label{rem:combinatorial-data}
The full $f$-vector and $h$-vector of the example above are
\[
  f(P)=(14,86,292,590,712,476,136),
  \qquad
  h(P)=(1,7,23,37,37,23,7,1).
\]
The five present cube-antipodal pairs are
\[
\begin{array}{lll}
\{1011101,0100010\}, & \{1000100,0111011\}, & \{1001010,0110101\},\\
\{0010001,1101110\}, & \{0001100,1110011\}. &
\end{array}
\]
Thus exactly five  antipodal pairs are present.  In agreement with
the discussion after \Cref{prop:ziegler-bound}, the graph has exactly five  non-edges and they are exactly these five pairs.  
\end{remark}

\section{The complete classification in dimension seven}
\label{sec:enumeration}

In \Cref{thm:main} we have presented one example that is not centrally symmetric. In this section we determine \emph{all} of them in dimension $7$ and as a by-product also confirm, for $d\le6$, the enumeration result that \citet{Ziegler1999} attributes to Aichholzer. 

\begin{definition}[Column-sum vector]
\label{def:colsum}
For a finite set $A\subseteq\bits^{n}$, regarded as the rows of a $0/1$-matrix
with $n$ columns, its \emph{column-sum vector} is $\sum_{a\in A}a$, the vector
in $\{0,1,\dots,|A|\}^{n}$ whose $j$th entry $\bigl|\{a\in A:a_j=1\}\bigr|$ is the
number of points of $A$ whose $j$th coordinate equals $1$.
\end{definition}

\paragraph{Reduction of search space.}
Let $P = \conv(V)$ be a simplicial 0/1-polytope with $|V|=2d$.  By
\Cref{prop:ziegler-bound} the vertex set is balanced and, for each coordinate,
the two coordinate faces are $(d-1)$-simplices.  Fixing the last coordinate
splits
\[
  V=(A\times\{0\})\;\cup\;(B\times\{1\}),\qquad A,B\subseteq\bits^{d-1},
\]
where $A$ and $B$ are the vertex sets of those two coordinate faces, respectively; each one is an affinely
independent set of $d$ points, i.e., the vertex set of a $(d-1)$-simplex.  Writing $B'=\one-B$ for
the complement of $B$ inside the smaller cube, balance is equivalent to $A$ and
$B'$ having the same column-sum vector, and $V$ is centrally symmetric if and
only if $A=B'$.  Thus the possible sets $V$ are encoded by pairs $(A,B')$ of
$0/1$-simplices in $\bits^{d-1}$ sharing a column-sum vector, and the
ones that are not centrally symmetric are exactly those examples with $A\neq B'$.

\paragraph{Symmetry reduction and certification.}
The cube symmetry group $B_{d-1}=(\mathbb Z_2)^{d-1}\rtimes S_{d-1}$ acts on the
first $d-1$ coordinates.  Grouping the simplices $A$ by their column-sum vectors and
retaining only a canonical representative of each orbit (with the entries of the column-sum vector folded into
$\{0,\dots,\lfloor d/2\rfloor\}$ and sorted, then reduced by the stabilizer of
the column-sum vector) realizes the full $B_{d-1}$ symmetry.  For $d=7$ this
shrinks the number of candidate pairs from $5.06\times10^{10}$ to
$4.26\times10^{8}$.  Each candidate is first screened by a cheap necessary test
(every coordinate face must be a $(d-1)$-simplex) and then certified by exact
integer facet enumeration; since all arithmetic is over the integers, the
classification is exact.  The examples that are not centrally symmetric are finally
deduplicated under the full cube group $B_d$.  The whole computation runs in a
few minutes on a single multi-core machine.  We validated the pipeline against a
direct enumeration of all $2d$-subsets for $d=4,5$, recovered the example of
\Cref{thm:main} for $d=7$, and re-verified every dimension-$7$ representative
independently in exact rational arithmetic and with \texttt{polymake}. For the centrally symmetric case, no separate search is needed, by the following folklore lemma:

\begin{lemma}
\label{lem:central-cross}
A centrally symmetric $d$-polytope with exactly $2d$ vertices is affinely
isomorphic to the $d$-dimensional cross polytope.
\end{lemma}

\begin{proof}
Let $P$ be centrally symmetric about $c$ with $2d$ vertices.  The central
reflection $x\mapsto 2c-x$ fixes no vertex, since its only fixed point $c$ is
in the interior, so it pairs the vertices into $d$ pairs $\{c+u_i,\,c-u_i\}$.  Since
$P$ is $d$-dimensional, the vectors $\{\pm u_i\}$ span $\R^d$, so
$u_1,\dots,u_d$ are linearly independent.  The linear map $e_i\mapsto u_i$ is
therefore invertible and carries the standard cross polytope
$\conv\{\pm e_i\}$ onto $P-c$.
\end{proof}

In particular every centrally symmetric example with $|V| = 2d$ is a $0/1$-realization of the cross polytope. We obtain the following classification:

\begin{theorem}[Classification]
\label{thm:classification}
Up to the symmetry group $B_d$ of the cube, the simplicial $d$-dimensional
$0/1$-polytopes with $2d$ vertices are counted in \Cref{tab:classification}.
For $d\le6$ every such polytope is a cross polytope; in dimension $7$ exactly
five of them are not centrally symmetric.
\end{theorem}

\begin{table}[htbp]
  \centering
  \begin{tabular}{rrrr}
    \toprule
    $d$ & total & cross polytopes & not centrally symmetric \\
    \midrule
    $4$ & $3$    & $3$    & $0$ \\
    $5$ & $7$    & $7$    & $0$ \\
    $6$ & $63$   & $63$   & $0$ \\
    $7$ & $1631$ & $1626$ & $5$ \\
    \bottomrule
  \end{tabular}
  \caption{Simplicial $0/1$-polytopes with the maximal number $2d$ of vertices,
  up to the cube symmetry group $B_d$.  By \Cref{lem:central-cross} the
  ``cross polytopes'' column counts the $0/1$-embeddings of the cross polytope;
  the last column counts the examples that are not centrally symmetric.  Dimension $7$ is the first one 
  with such an example.}
  \label{tab:classification}
\end{table}

The five dimension-$7$ examples that are not centrally symmetric fall into two combinatorial types, separated already by their face numbers (\Cref{tab:d7-types}).  All five are balanced with five cube-antipodal pairs and four unmatched vertices,
exactly as in \Cref{rem:combinatorial-data}.  Four of them share the $f$- and
$h$-vector of the example in \Cref{thm:main}; they are pairwise not equivalent
under the cube group but combinatorially isomorphic, so they form a
single combinatorial type realized as four distinct $0/1$-embeddings, verified
with \texttt{polymake}.  They are not all \emph{affinely} isomorphic, however:
as point configurations the four embeddings realize three distinct
affine-isomorphism types (two of them coincide affinely, the other two are
affinely distinct from these and from each other).  The fifth example is different: it has $144$
facets and the entry $h_3=41$ is strictly larger, hence $g_3=h_3-h_2=18$ instead
of $14$.  An explicit representative of this $144$-facet type is
\[
\begin{aligned}
V^\ast=\{\,&0000000,\ 1100000,\ 1010000,\ 0001100,\ 1111100,\ 1001010,\ 0110110,\\
          &1001001,\ 0110101,\ 1110011,\ 0111011,\ 1000111,\ 0101111,\ 0011111\,\}.
\end{aligned}
\]
Explicit generating vertex sets for all five orbits as well as additional information are to be found in
\Cref{tab:d7-examples} in the appendix.

\begin{table}[htbp]
  \centering
  \small
  \begin{tabular}{clcll}
    \toprule
    type & \multicolumn{1}{c}{orbits} & facets &
      \multicolumn{1}{c}{$f$-vector} & \multicolumn{1}{c}{$h$-vector} \\
    \midrule
    I  & $1$ & $144$ & $(14,86,296,610,748,504,144)$ & $(1,7,23,41,41,23,7,1)$ \\
    II & $4$ & $136$ & $(14,86,292,590,712,476,136)$ & $(1,7,23,37,37,23,7,1)$ \\
    \bottomrule
  \end{tabular}
  \caption{The two combinatorial types of not centrally symmetric simplicial $7$-dimensional
  $0/1$-polytope with $14$ vertices.  Type~II contains the example of
  \Cref{thm:main}; both types have exactly five cube-antipodal pairs, i.e. non-edges.}
  \label{tab:d7-types}
\end{table}

\section{Concluding questions}
\label{sec:questions}

The classification settles the existence and number of not centrally symmetric simplicial polytopes with $|V| = 2d$ for dimension $d = 7$. Several questions remain.

\begin{enumerate}[leftmargin=*]
  \item How does the number of not centrally symmetric examples grow for $d\ge8$?  
  \item All five  examples with $d=7$ have $d-2$  cube-antipodal pairs. Which numbers of cube-antipodal pairs occur in for $d\ge 8$? Are there simplicial 0/1-polytopes without any cube-antipodal pairs, i.e., are there any $2$-neighborly simplicial 0/1-polytopes?
  \item Is there a structural characterization (not relying on enumeration) of the split
  pairs $(A,B')$ that yield a simplicial polytope?
\end{enumerate}

\section*{Acknowledgments}

The research reported in this paper was initiated at a live presentation of an internal research version of the \href{https://www.pokutta.com/blog/agentic-researcher/}{Agentic Researcher Framework} (an LLM-based agentic research framework; see \citet{ZPRP2026AgenticResearcherAI4Research}) in Magdeburg, Germany in June, where the conjecture of Ziegler was used as an example. After an initially wrong affirmative conclusion by the AI agent via a flawed proof, a subsequent Lean 4 verification attempt by the agent revealed a gap in the argument\footnote{In the flawed proof, one of the inequalities in the $g$-theorem (\citet{McMullen1971,Stanley1975,BilleraLee1980}, see also \citet[Thm. 8.5]{Ziegler2007}) was modified resulting, for $d=7$, in the bound $g_2=h_2-h_1\le 15$, where the example that subsequently was derived has $g_2=16$.}.
Over multiple iterations the agent further reduced the gaps to a combinatorial condition on the cube-antipodal pair structure of the equality case (the split-model reduction underlying \Cref{sec:enumeration}), which gave the insight that led to the counterexample. The correctness of the counterexample has been verified via \texttt{polymake} and separate independent implementation in exact rational arithmetic. In a following iteration, the complete classification for $d=7$ has been derived, via enumeration exploiting symmetry and balancedness. The LLMs used in the agentic harness were a locally-deployed DeepSeek V4 Flash together with GLM 5.2 with a success/failure-based feedback loop via reflection prompts, similar to GEPA \citep{agrawal2026gepareflectivepromptevolution}. We report the agentic research milestones in \Cref{app:milestones}. 

Research reported in this paper was partially supported by project \emph{EF-LI-Opt-3: Agentic AI in Mathematics} of the Berlin Mathematics Research
Center MATH$^+$ (EXC-2046/2, project ID 390685689), funded by the Deutsche
Forschungsgemeinschaft (DFG, German Research Foundation) under Germany's
Excellence Strategy.

\bibliographystyle{plainnat}
\bibliography{refs}

\clearpage 

\appendix
\section{The five not centrally symmetric examples in dimension seven}
\label{app:examples}

\Cref{tab:d7-examples} gives explicit generating vertex sets for the five
simplicial, not centrally symmetric $7$-dimensional $0/1$-polytopes with $14$ vertices of \Cref{thm:classification}, one representative per orbit under the cube symmetry
group $B_7$.

\begin{table}[htbp]
\centering\small
\setlength{\tabcolsep}{8pt}
\begin{tabular}{>{\ttfamily}c >{\ttfamily}c >{\ttfamily}c >{\ttfamily}c >{\ttfamily}c}
\toprule
\multicolumn{1}{c}{$V_1$} & \multicolumn{1}{c}{$V_2$} &
\multicolumn{1}{c}{$V_3$} & \multicolumn{1}{c}{$V_4^{\dagger}$} &
\multicolumn{1}{c}{$V_5$} \\
\midrule
0000000 & 0000000 & 0000000 & 0000000 & 0000000 \\
1100000 & 1100000 & 1100000 & 1100000 & 1100000 \\
1010000 & 1010000 & 1010000 & 1010000 & 0011000 \\
0001100 & 0001100 & 0001100 & 0001100 & 1010100 \\
1111100 & 1001010 & 1001010 & 1001010 & 0101100 \\
1001010 & 0111010 & 0111010 & 0111010 & 0110010 \\
0110110 & 1000110 & 0100110 & 0110110 & 1001110 \\
1001001 & 0111001 & 1011001 & 1001001 & 0110001 \\
0110101 & 1000101 & 1000101 & 0110101 & 1001101 \\
1110011 & 1111101 & 0111101 & 1111101 & 1010011 \\
0111011 & 1110011 & 1110011 & 1110011 & 1101011 \\
1000111 & 0110111 & 0010111 & 1000111 & 0100111 \\
0101111 & 0101111 & 0101111 & 0101111 & 0011111 \\
0011111 & 0011111 & 1111111 & 0011111 & 1111111 \\
\bottomrule
\end{tabular}
\caption{Generating vertex sets of the five not centrally symmetric, simplicial
$7$-dimensional $0/1$-polytopes with $2d=14$ vertices (\Cref{thm:classification}),
one per cube-symmetry orbit.  Each column lists the $14$ vertices of one
representative as bit strings, with coordinate $j$ the $j$th digit.  $V_1$ is
the $144$-facet type (the example $V^\ast$ of \Cref{sec:enumeration}); the
remaining $V_2,\dots,V_5$ are the $136$-facet type, and $V_4^{\dagger}$ is the
orbit of the counterexample in \Cref{thm:main}.  All five are balanced with
five cube-antipodal pairs.  Although $V_2,\dots,V_5$ are combinatorially
isomorphic, as point configurations they realize three affine-isomorphism
types: $V_2$ and $V_3$ are affinely isomorphic, whereas $V_4^{\dagger}$ and
$V_5$ are affinely distinct, both from each other and from $V_2,V_3$.}
\label{tab:d7-examples}
\end{table}

The single coordinate swap of \Cref{rem:single-swap}, exchanging one coordinate
between two vertices, again distinguishes the two types.  For each $136$-facet
configuration a single such swap already restores central symmetry, and in fact
exactly four distinct single swaps do, each producing the $7$-cross polytope by
\Cref{lem:central-cross}.  The $144$-facet configuration $V_1$ admits no single
swap of this kind.  \Cref{tab:d7-swaps} lists the minimum number of swaps needed
for each configuration.

\begin{table}[htbp]
  \centering
  \begin{tabular}{ccrr}
    \toprule
    configuration & type & facets & min.\ swaps to central symmetry \\
    \midrule
    $V_1$           & I  & $144$ & $2$ \\
    $V_2$           & II & $136$ & $1$ \\
    $V_3$           & II & $136$ & $1$ \\
    $V_4^{\dagger}$ & II & $136$ & $1$ \\
    $V_5$           & II & $136$ & $1$ \\
    \bottomrule
  \end{tabular}
  \caption{The minimum number of coordinate swaps (each exchanging one
  coordinate between two vertices) needed to make each configuration centrally
  symmetric.  For every $136$-facet configuration a single swap suffices, and
  exactly four distinct single swaps do, each producing the $7$-cross polytope;
  the $144$-facet configuration $V_1$ admits no single swap and needs two.}
  \label{tab:d7-swaps}
\end{table}

\begin{remark}
\label{rem:type-I-two-swaps}
For the $144$-facet configuration $V_1$ a single swap never suffices, but two
do: the minimum number of coordinate swaps that make $V_1$ centrally symmetric
is exactly two, and two suitable swaps produce the $7$-cross polytope.  In this
sense $V_1$ is strictly farther from central symmetry than the four $136$-facet
types, consistent with its larger $h$-vector.
\end{remark}

\section{Agentic research milestones}
\label{app:milestones}

\Cref{fig:milestones} places the principal milestones of the agentic research
process on a wall-clock timeline; each bar spans the
work turn in which the corresponding milestone was reached as well as the lead-up time.

\begin{figure}[htbp]
  \centering
  \includegraphics[width=\linewidth]{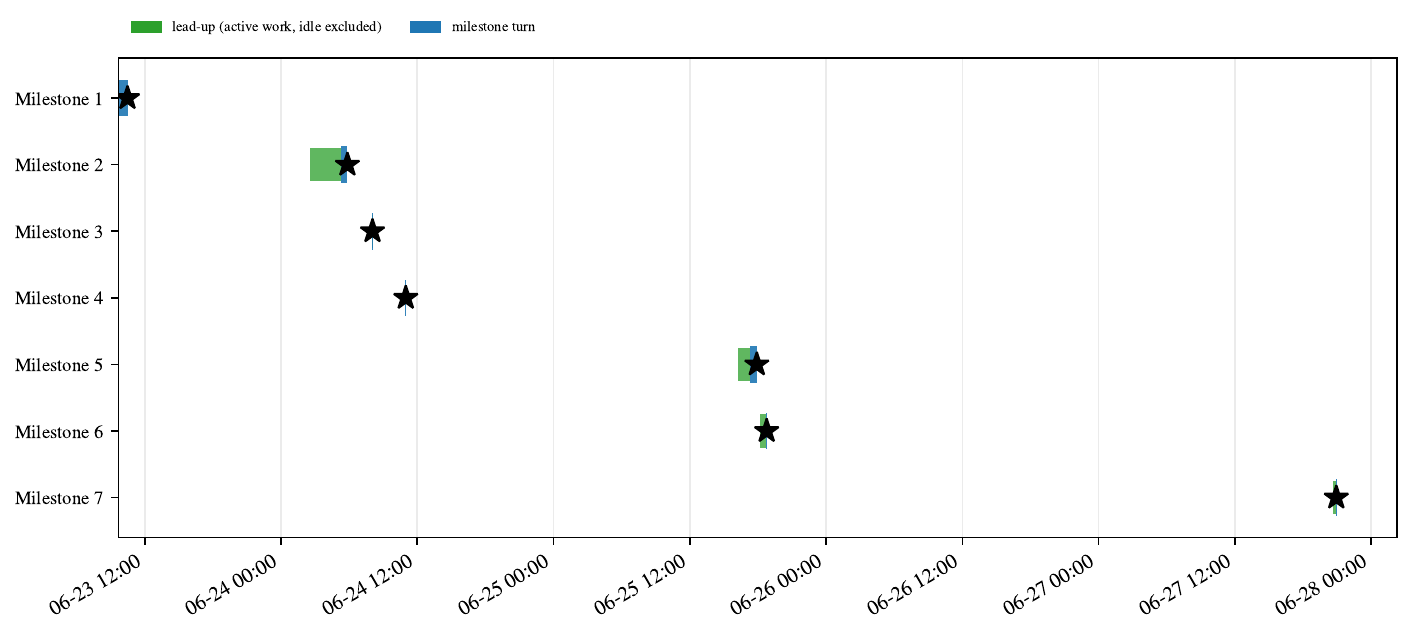}
  \caption{Milestones over time. Each row shows, in green, the actual work
  leading up to the milestone (the intermediate steps, with idle gaps excluded),
  followed by the milestone's own work turn in colour, star-capped at completion;
  durations below are reported as (lead-up\,$+$\,turn) in minutes.
  \textbf{Milestone~1} ($3+45$): a flawed argument ``established'' the
  conjecture as true via a Macaulay/$g$-theorem bound, later found invalid.
  \textbf{Milestone~2} ($161+36$): the counterexample, a not centrally symmetric
  simplicial $7$-dimensional $0/1$-polytope with $14$ vertices, was found.
  \textbf{Milestone~3} ($3+1$): it was verified exactly in rational arithmetic.
  \textbf{Milestone~4} ($0+4$): it was independently verified with
  \texttt{polymake}. \textbf{Milestone~5} ($64+35$): the complete dimension-$7$
  classification (exactly five not centrally symmetric examples) was established
  by exact exhaustive enumeration. \textbf{Milestone~6} ($31+5$): a single
  coordinate swap was shown to turn each $136$-facet example into the cross
  polytope. \textbf{Milestone~7} ($15+2$): the five examples were shown to form
  exactly two combinatorial types, using \texttt{polymake}.}
  \label{fig:milestones}
\end{figure}

\end{document}